\documentclass[12pt,a4paper]{article}
\RequirePackage{amsthm,amsmath,amssymb,amsfonts}

\theoremstyle{definition}

\newtheorem{remark}{Remark}
\theoremstyle{plain}
\newtheorem{theorem}{Theorem}
\newtheorem{lemma}{Lemma}
\newtheorem{corollary}{Corollary}
\author{Shoou-Ren Hsiau  and Yi-Ching Yao\\
National Changhua  University of Education and Academia Sinica}

\title{Stochastic ordering for Bernoulli and normal random walks}

\begin{document}

\newcommand{\ML}{\mathcal{L}}
\newcommand{\MP}{\mathbb{P}}
\newcommand{\MS}{\mathbb{S}}
\newcommand{\MST}{\widetilde{\mathbb{S}}}
\newcommand{\MSTU}{\widetilde{\mathbb{S}^{+}}}
\newcommand{\MSTD}{\widetilde{\mathbb{S}^{-}}}
\newcommand{\MSO}{\widehat{\mathbb{S}}}
\newcommand{\MSOU}{\widehat{\mathbb{S}^+}}
\newcommand{\MSOD}{\widehat{\mathbb{S}^-}}
\newcommand{\HT}{\widehat{T}}
\newcommand{\CS}{\mathcal{S}}
\newcommand{\MM}{\mathcal{M}}
\newcommand{\MG}{\mathcal{G}}
\newcommand{\half}{\frac{1}{2}}

\maketitle
\begin{abstract}
Let  $(S_n^p)_{n\geq 0}$ be a Bernoulli random walk where each of the independent increments is 
either $1$ or $-1$ with probabilities $p$ and $1-p$. For $p'$ and $p'' \in [0,1]$ with
 $|p'-\half|>|p''-\half|$, we show that $(|S_n^{p''}|)_{n\geq 0}$ is stochastically smaller
 than $(|S_n^{p'}|)_{n\geq 0}$. In other words,  $(|S_n^{p}|)_{n\geq 0}$ is stochastically
 decreasing in $p \in [0,\half]$ and increasing in $p\in [\half,1]$. An analogous result is also given
 for the family of normal random walks indexed by $\mu \in R$ where each of the independent increments
  is normally distributed with common mean
 $\mu$ and variance $1$. Extension to Brownian motion then follows by a limiting argument.
 As an application, these results are used to easily derive stochastic
 ordering properties for stopping times of Bernoulli and normal random walks.


\end{abstract}





\section{Introduction and main results}

We are concerned here with stochastic ordering properties for Bernoulli and normal random walks.
For $p \in [0,1]$, let $(S_n^{p})_{n \geq 0}$ be a Bernoulli random walk where
 $S_0^{p}=0, S_n^{p}=X_1^{p}+\cdots+X_n^{p}\;(n\geq 1)$, and the increments $X_1^{p}, X_2^{p}, \dots$ are independent and identically distributed (i.i.d.) with $P(X_n^{p}=1)=p$ and $P(X_n^{p}=-1)=q=1-p$.
We write  $\textbf{U} \overset{d}{=} \textbf{V}$ to mean that $\textbf{U}$ and $\textbf{V}$ 
have the same distribution where 
$\textbf{U}$ and $\textbf{V}$ are finite- or infinite-dimensional random vectors. For infinite-dimensional 
$\textbf{U}=(U_n)_{n \geq 0}$ and $\textbf{V}=(V_n)_{n\geq 0}$,
 we say that $\textbf{U}$ is stochastically smaller than $\textbf{V}$ and write 
$\textbf{U} \preceq_{st} \textbf{V}$
if $Ef(\textbf{U})\leq Ef(\textbf{V})$
for all bounded increasing measurable functions $f: R^{\infty} \to R$,
where $f$ is increasing (or more precisely, non-decreasing) if 
$f(x_0,x_1,\dots)\leq f(y_0, y_1,\dots)$ whenever $x_i\leq y_i$ for all $i$. 
(The infinite-dimensional  space $R^\infty$ is endowed with the product topology and
  measurability is with respect to the Borel $\sigma-\text{field}$.) 
Likewise, $\textbf{U} \preceq_{st} \textbf{V}$ is similarly defined
for finite-dimensional $\textbf{U}$ and $\textbf{V}$. 
Thus, the (integral) stochastic order $\preceq_{st}$ is
defined with respect to the class of bounded increasing measurable functions.
(See \cite[Chapter 2]{MS} for a comprehensive discussion
of integral stochastic orders.)
 For $p+p'=1$, since $X_n^{p} \overset{d}{=} -X_n^{p'}$, we have 
$(S_n^{p})_{n \geq 0} \overset{d}{=}(-S_n^{p'})_{n \geq 0}$. So,
$(|S_n^{p}|)_{n\geq 0} \overset{d}{=}(|S_n^{p'}|)_{n\geq 0}$ for $p+p'=1$.
Theorem \ref{thm1} below shows that  $(|S_n^p|)_{n \geq 0}$ is  stochastically
decreasing in $ p \in [0,\half]$ and increasing in $p \in [\half, 1]$.

\begin{theorem}\label{thm1}
For $p', p'' \in [0,1]$ satisfying $|p'-\half|> |p''-\half|$, we have $(|S_n^{p''}|)_{n\geq 0} 
\preceq_{st} (|S_n^{p'}|)_{n\geq 0}$.
\end{theorem}

We further consider two  closely related random walks  $(U_n^p)_{n \geq 0}$ and $(V_n^p)_{n \geq 0}$ 
with a random initial state 
defined as follows. Let $X_0 \in \{+1,-1\}$ be independent of $X_1^p,X_2^p,\dots$, such that $X_0$ is either $+1$ or $-1$ with equal probabilities.
Let $U_0^p=V_0^p=X_0$ and for $n\geq 1$,
\begin{align*}
U_n^p&=X_0+X_1^p+\dots+X_n^p=X_0+S_n^p,\\
V_n^p&=X_0+2X_1^p+\cdots+2X_n^p=X_0+2S_n^p.
\end{align*}
Note that for $p+p'=1$, we have
\[ (|U_n^p|)_{n\geq 0} \overset{d}{=}(|U_n^{p'}|)_{n\geq 0}\;\;\;\text{and}\;\;\;
(|V_n^p|)_{n\geq 0} \overset{d}{=}(|V_n^{p'}|)_{n\geq 0}.
\]
Note also that $V_n^p$ is odd  for all $n$ with probability 1.  
\begin{theorem}\label{thm2}
For $p'=\half \neq p \in [0,1]$, we have
$(|U_n^{p'}|)_{n\geq 0} \preceq_{st} (|U_n^{p}|)_{n\geq 0}$.
Consequently, $(|U_n^{p}|)_{n\geq 0}$ is stochastically minimized over
$ p \in [0,1]$ by $p=\half$.
\end{theorem}

\begin{theorem}\label{thm3}
For $p', p'' \in [0,1]$ satisfying $|p'-\half|>|p''-\half|$, we have
$(|V_n^{p''}|)_{n\geq 0} \preceq_{st} (|V_n^{p'}|)_{n\geq 0}$.
Consequently,  $(|V_n^{p}|)_{n\geq 0}$ is stochastically decreasing
in $p \in [0,\half]$ and  increasing in $p \in [\half,1]$.
\end{theorem}

Next we present an analogous result for   normal random walks. Let  $(W_n^{\mu})_{n\geq 0}$ be a 
normal random  walk defined by  
$W_0^{\mu}=0$ and
$W_n^{\mu}=Z_1^\mu+\cdots+Z_n^\mu \;(n\geq 1)$, where the independent increments $Z_1^\mu, Z_2^\mu,\dots$
are normally distributed with common mean $\mu$ and variance $1$.
Note that $(|W_n^\mu|)_{n\geq 0} \overset{d}{=} (|W_n^{-\mu}|)_{n\geq 0}$. 
\begin{theorem}\label{thm4}
For $\mu', \mu'' \in R$ satisfying $|\mu'|<|\mu''|$, we have
 $(|W_n^{\mu'}|)_{n\geq 0} \preceq_{st} (|W_n^{\mu''}|)_{n\geq 0}$.
Consequently, $(|W_n^\mu|)_{n\geq 0}$  is stochastically decreasing
in $\mu \in (-\infty, 0]$ and  increasing in $\mu \in [0,\infty)$.
\end{theorem}

The next four corollaries are simple consequences of  Theorems \ref{thm1}--\ref{thm4}.

\begin{corollary}\label{cor1}
For integers $b_n \geq 0 \;(n\geq 1)$, define the stopping time
 $T_{b_1,b_2,\dots}^{p}=\inf\{n \geq 1: |S_n^p|\geq b_n\}$.
Then $T_{b_1,b_2,\dots}^p$
is stochastically increasing in $p \in [0,\half]$ and decreasing in $p \in [\half,1]$.
\end{corollary}

\begin{corollary}\label{cor2}
For integers $b_n \geq 0 \;(n\geq 1)$, let $\HT_{b_1,b_2,\dots}^{p}=\inf\{n \geq 1: |U_n^p|\geq b_n\}$.
Then $\HT_{b_1,b_2,\dots}^{p}$
is stochastically maximized over $p \in [0,1]$ by $p=\half$.
\end{corollary}

\begin{corollary}\label{cor3}
For odd integers $b_n \geq 1 \;(n\geq 1)$, let $\widetilde{T}_{b_1,b_2,\dots}^p=\inf\{n \geq 1: |V_n^p|\geq b_n\}$.
Then $\widetilde{T}_{b_1,b_2,\dots}^p$
is stochastically increasing in $p \in [0,\half]$ and decreasing in $p \in [\half,1]$.
\end{corollary}

\begin{corollary}\label{cor4}
For real $b_n\geq 0\;(n\geq 1)$, let $\overline{T}_{b_1,b_2,\dots}^\mu=\inf\{n\geq 1:
|W_n^\mu|\geq b_n\}$. Then $\overline{T}_{b_1,b_2,\dots}^\mu$ is stochastically increasing in
$\mu \in (-\infty,0]$ and decreasing in $\mu \in [0,\infty)$.
\end{corollary}

To see that Corollary \ref{cor1} follows from Theorem \ref{thm1}, note that for $n\geq 1$,
the function $f:R^{\infty} \to R$ given by
\begin{align*}
f(x_0,x_1,\dots)=
\begin{cases}
-1, &\text{if $x_i< b_i, i=1,\dots, n$;}\\
0, &\text{otherwise;}
\end{cases}
\end{align*}
is  bounded increasing. By Theorem \ref{thm1}, 
\[
Ef(|S_0^p|, |S_1^p|,\dots)=-P(|S_i^p|< b_i, i=1,\dots, n)=-P(T_{b_1,b_2,\dots}^p>n)
\]
is decreasing in $p \in [0,\half]$ and increasing in $p \in [\half,1]$, implying that
$T_{b_1,b_2,\dots}^p$ is stochastically increasing in $p \in [0, \half]$ and decreasing in $p \in [\half,1]$.
The same argument also applies to  Corollaries \ref{cor2}--\ref{cor4}.
The following stronger version of Corollary \ref{cor2} requires a 
somewhat long and delicate argument.

\begin{theorem}\label{thm5}
For integers $b_n \geq 0 \;(n\geq 1)$, 
  $\HT_{b_1,b_2,\dots}^p$
is stochastically increasing in $p \in [0,\half]$ and decreasing in $p \in [\half,1]$.
\end{theorem}

The proofs of Theorems \ref{thm1}--\ref{thm4} are given in Section 2 and that of Theorem  \ref{thm5}
 is given in Section 3.
We close this section with a number of remarks discussing the implications of the above results
and an extension to Brownian motion.

\begin{remark}
Theorems \ref{thm1}--\ref{thm4} are proved by verifying the condition of Lemma \ref{lem2} in Section 2.
Although the random walks $(S_n^p)_{n\geq 0}, (U_n^p)_{n \geq 0}$ and $(V_n^p)_{n\geq 0}$ are 
closely related, compared with Theorems \ref{thm1} and \ref{thm3} for $(|S_n^p|)_{n\geq 0}$ and
$(|V_n^p|)_{n\geq 0}$, 
we are unable to establish the stronger version of Theorem \ref{thm2}  that 
$(|U_n^{p''}|)_{n\geq 0} \preceq_{st} (|U_n^{p'}|)_{n\geq 0}$ for
$p', p'' \in [0,1]$ satisfying $|p'-\half|>|p''-\half|$ 
(which would imply Theorem \ref{thm5}). This is due to the fact that 
the condition of Lemma \ref{lem2}  does not hold in this case. 
Furthermore, unlike $(|S_n^p|)_{n\geq 0}, (|V_n^p|)_{n\geq 0}$
and $(|W_n^\mu|)_{n\geq 0}$, $(|U_n^{p}|)_{n\geq 0}$ is not a Markov chain for $\half \neq p \in (0,1)$.
For more detail, see Remark \ref{rem7} in Section 2.
\end{remark}

\begin{remark}
The stochastic ordering properties for the stopping times as described in Corollaries \ref{cor1}--\ref{cor4}
follow immediately from Theorems \ref{thm1}--\ref{thm4}. By contrast, in the literature, 
these properties are derived by more delicate arguments based on induction (similar
to the proof of Theorem \ref{thm5}). Below we review some of the works in the literature. 
\end{remark}

\begin{remark}
The stopping time $T_{b_1,b_2,\dots}^p$ as defined in Corollary \ref{cor1} arises  in 
Bayesian sequential testing for Bernoulli trials (\cite{MR, SW}) and the Kiefer-Weiss problem
(\cite{KW, NNF, Weiss}). 
Given  $X_1^p,X_2^p,\dots$ observed sequentially (with unknown $p$), 
Moriguti and Robbins \cite{MR} and Simons and Wu \cite{SW} considered testing the null hypothesis $p\leq \half$ versus the alternative $p>\half$ subject to
a unit cost for each observed random variable and an additional cost $C \;|p-\half|$ if the wrong
hypothesis is chosen where $C>0$ is constant. They adopted a Bayesian approach with a Beta prior 
Beta$(c,c')$ imposed
on $p$ 
 with $c>0$ and $c'>0$. It follows from
their results that
for the symmetric case $c=c'$, the Bayesian optimal stopping rule is 
$T_{b_1,b_2,\dots}^p$  where $b_1,b_2,\dots$ depend on $C$ and 
$c$ and can be found by backward induction.
The related Kiefer-Weiss problem
in the Bernoulli case is to find, for given $0<p_1<p_2<1$,  a sequential test of $p=p_1$ versus $p=p_2$, with
preassigned type I and type II error probabilities $\alpha>0$ and $\beta>0$, and which
minimizes the maximum expected sample size (the maximum taken with respect to all $p \in [0,1]$).
For the symmetric case with $p_1+p_2=1$ and $\alpha=\beta$,  Weiss \cite{Weiss} showed that the stopping rule of the
minimax solution is exactly
$T_{b_1,b_2,\dots}^p$ for some $b_1,b_2,\dots$. 
Although he claimed only that $E(T_{b_1,b_2,\dots}^p)$
is maximized over $p \in [0,1]$ by $p=\half$, his induction-based argument actually establishes
the stronger result as stated in Corollary \ref{cor1}. 
Furthermore, it may be worth noting that Corollary \ref{cor1} can be 
strengthened with the usual stochastic order replaced
by the likelihood ratio order provided $\{b_n\}$ is decreasing (i.e. $0\leq b_{n+1}\leq b_n$ for all $n$). 
This result is stated in the following theorem whose proof is given in Section 3.

\begin{theorem}\label{thm6}
 Let $b_n\geq 0$ be non-negative integers satisfying 
$0\leq b_{n+1}\leq b_n$ for all $n\geq 1$. Then in terms of the likelihood ratio order, 
$T_{b_1,b_2,\dots}^{p}$ is increasing in $p \in [0, \half]$ and decreasing in $p \in [\half, 1]$.
\end{theorem}
\end{remark} 

\begin{remark}
Weiss \cite{Weiss} also considered the Kiefer-Weiss problem for the normal case where
$Z_1^\mu, Z_2^\mu,\dots$ are observed sequentially with an unknown mean $\mu$. 
To test the null hypothesis $\mu=\mu_1$
versus the alternative $\mu=\mu_2$ with preassigned type I and type II error probabilities
$\alpha$ and $\beta$, Weiss showed for the symmetric case $\mu_1+\mu_2=0$ and $\alpha=\beta$
that the stopping rule of the minimax solution is $\overline{T}_{b_1,b_2,\dots}^{\mu}$ as defined in
Corollary \ref{cor4} for some $b_1,b_2,\dots$. He referred to \cite{Anderson2} for a proof that
$\overline{T}_{b_1,b_2,\dots}^{\mu}$ is stochastically maximized over $\mu \in R$ by $\mu=0$
(with the help of a useful probability inequality for Gaussian processes in 
\cite[Corollary 5]{Anderson1}).
\end{remark}

\begin{remark}
The classical gambler's ruin problem is studied in detail in \cite[Chapter XIV]{feller}, where 
for integers $0<z<a$, the gambler
plays an adversary with initial capitals $z$ and $a-z$, respectively. At each round,
the gambler wins or loses one dollar with probabilities $p$ and $q=1-p$, respectively. The game
continues until one of the players is ruined. The distribution of the duration of the game
depends on $p, a$ and $z$ and  admits
a complicated closed-form expression; see \cite[(5.7) and (5.8), pp.353--354]{feller}.
It is shown in \cite{PR}   that   the duration is stochastically maximized over $p\in [0,1]$ by $p=\half$, if 
 $a$ is  even and $z=\frac{a}{2}$ (i.e. the  initial capitals of the two players are equal). It is
 readily seen that the duration in this 
 case is $T_{b_1,b_2,\dots}^p$ with $b_n=\frac{a}{2}$ for all $n$.
Furthermore, it is noted in \cite{HY} that this result can be strengthened with the usual stochastic order replaced by the likelihood
ratio order. More precisely, in terms of the likelihood ratio order,  $T_{b_1,b_2,\dots}^p$ 
with $b_n=\frac{a}{2}$ for all $n$ is increasing in $p \in [0,\half]$ and decreasing in $p \in [\half, 1]$.
See also
\cite{PR1} for related results. If the pair of initial capitals 
$(z,a-z)$ is assumed to be either
$(\frac{a}{2}-1,\frac{a}{2}+1)$ or $(\frac{a}{2}+1,\frac{a}{2}-1)$
with equal probabilities (i.e. $z=\frac{a}{2} \pm 1$ each with probability $\half$), 
then the duration  is $\HT_{b_1,b_2,\dots}^p$ with $b_n=\frac{a}{2}$ for all $n$. By Corollary \ref{cor2}, the duration is stochastically maximized 
over $p \in [0,1]$ by $p= \half$.
 On the other hand,  if  $a$ is odd,  an analogous case is to assume that
 the
pair of initial capitals 
$(z,a-z)$ is either
$(\frac{a-1}{2},\frac{a+1}{2})$ or $(\frac{a+1}{2},\frac{a-1}{2})$
with equal probabilities. That is, the initial state $z$ is either 
$\frac{a-1}{2}$ or $\frac{a+1}{2}$ each with probability $\half$.
It is conjectured in \cite[Remark 4]{HY} that 
the duration in this case is stochastically maximized over $p \in [0,1]$ by $p=\half$.
In fact, the duration is $\widetilde{T}_{b_1,b_2,\dots}^p$ with $b_n=a$ for all $n$.
To see this,  consider a new (equivalent) version of  the problem as follows. 
At each round the gambler wins or loses 2 dollars with probabilities $p$ and $q$. The initial
capitals of the two players are $z$ and $2a-z$ where  $a$ is odd and $z$ is either $a-1$ or $a+1$ with equal probabilities. It is readily seen that for the original and new versions of the problem,
the duration is the same, which
 is exactly
$\widetilde{T}_{b_1,b_2,\dots}^p$ with $b_n=a$ for all $n$. By Corollary \ref{cor3}, the duration is
stochastically maximized over $p \in [0,1]$ by $p=\half$.  (It is this odd $a$ case above that motivated
us to consider the random walk $(V_n^p)_{n\geq 0}$.)
\end{remark}

\begin{remark}
By Theorem \ref{thm4}, we have $(|W_n^{\mu'}|)_{n\geq 0} \preceq_{st} (|W_n^{\mu''}|)_{n\geq 0}$
for $0\leq \mu'<\mu''$. It follows by Strassen's theorem on stochastic domination 
 (\emph{cf.} \cite{Lindvall, Strassen}) that  
  there exists  a coupling $((\Delta_n)_{n\geq 0}, (\Gamma_n)_{n\geq 0}))$ 
  of $(|W_n^{\mu'}|)_{n\geq 0}$ and  $(|W_n^{\mu''}|)_{n\geq 0}$
  such that
  $\Delta_n \leq \Gamma_n$ for all $n$ a.s. (More precisely, 
   $(\Delta_n)_{n\geq 0}$ and $(\Gamma_n)_{n\geq 0}$
  are defined on a common probability space such that $(\Delta_n)_{n\geq 0}\overset{d}{=}
  (|W_n^{\mu'}|)_{n\geq 0}$, $(\Gamma_n)_{n\geq 0}\overset{d}{=}
  (|W_n^{\mu''}|)_{n\geq 0}$, and $\Delta_n \leq \Gamma_n$ for all $n$ a.s.)
We can extend Theorem \ref{thm4} to Brownian motion by a standard limiting argument (\emph{cf.} \cite{Billingsley}).
Let $(B_t)_{t\geq 0}$ be standard Brownian motion and let $B_t^{\mu}=B_t+\mu t$ be Brownian motion
with drift parameter $\mu$. For $0\leq \mu'<\mu''$, we claim that
$(|B_t^{\mu'}|)_{0\leq t \leq 1} \preceq_{st} (|B_t^{\mu''}|)_{0\leq t \leq 1}$.
Here the (Polish) space $C[0,1]$ is endowed with the partial 
ordering $\leq^*$ defined by 
\begin{align*}
``x \leq^* y\;\; \text{if and only if $x(t)\leq y(t)$ for all
$t \in [0,1]$.''}
\end{align*}
Let $\MM_n$ be  a continuous mapping from $R^n$ to $C[0,1]$ such that 
$\MM_n\big[(x_i)_{1\leq i\leq n}\big]=x \in C[0,1]$ where $x$ is a piecewise linear function given by
$x(0)=0,\; x(\frac{i}{n})=x_i, i=1,\dots, n$, and
\begin{align*}
x(t)&=(i-nt)\; x(\frac{i-1}{n})+(nt-i+1)\; x(\frac{i}{n})\; 
\;\text{for $\frac{i-1}{n} <t <\frac{i}{n},\;\; i=1,\dots, n$}.
\end{align*}
Since $(\frac{1}{\sqrt{n}} W_i^{\mu/\sqrt{n}})_{1\leq i \leq n}\overset{d}{=}
(B_{i/n}^{\mu})_{1\leq i\leq n}$, we have
$(\frac{1}{\sqrt{n}} |W_i^{\mu/\sqrt{n}}|)_{1\leq i \leq n}\overset{d}{=}
(|B_{i/n}^{\mu}|)_{1\leq i\leq n}$ and
\begin{align}\label{eq09096}
\MM_n\big[(\frac{1}{\sqrt{n}} |W_i^{\mu/\sqrt{n}}|)_{1\leq i \leq n}\big]
\overset{d}{=}
\MM_n\big[(|B_{i/n}^{\mu}|)_{1\leq i\leq n}\big],
\end{align}
where both sides of (\ref{eq09096}) are random elements of $C[0,1]$. 
Moreover, as $n \to \infty$,
the right side of (\ref{eq09096}) converges a.s. to  $(|B_t^{\mu}|)_{0\leq t \leq 1}$,
implying that 
\begin{align}\label{eq09097}
\MM_n\big[(\frac{1}{\sqrt{n}} |W_i^{\mu/\sqrt{n}}|)_{1\leq i \leq n}\big]
\;\;\text{converges weakly to 
$(|B_t^{\mu}|)_{0\leq t \leq 1}$\;.}
\end{align}
By Theorem \ref{thm4}, for $0\leq \mu'<\mu''$,
$(\frac{1}{\sqrt{n}} |W_i^{\mu'/\sqrt{n}}|)_{1\leq i \leq n} \preceq_{st} 
(\frac{1}{\sqrt{n}} |W_i^{\mu''/\sqrt{n}}|)_{1\leq i \leq n}$, so 
there exists a coupling $\big((Y_i')_{1\leq i \leq n}, (Y_i'')_{1\leq i \leq n}\big)$
of $(\frac{1}{\sqrt{n}} |W_i^{\mu'/\sqrt{n}}|)_{1\leq i \leq n}$ and  
$(\frac{1}{\sqrt{n}} |W_i^{\mu''/\sqrt{n}}|)_{1\leq i \leq n}$ such that
$Y_i'\leq Y_i''$ for $1\leq i \leq n$ a.s. It follows that
\begin{align}\label{eq09091}
\MM_n\big[(Y_i')_{1\leq i \leq n}\big]\leq^*  \MM_n\big[(Y_i'')_{1\leq i \leq n}\big]
\; \text{a.s.}\;,
\end{align}
 where both sides  are random elements of $C[0,1]$ defined on a common probability space.
Since 
\begin{align*}
(Y_i')_{1\leq i \leq n}\overset{d}{=} (\frac{1}{\sqrt{n}} |W_i^{\mu'/\sqrt{n}}|)_{1\leq i \leq n}\;\;
\text{and}\;\; (Y_i'')_{1\leq i \leq n}\overset{d}{=} (\frac{1}{\sqrt{n}} |W_i^{\mu''/\sqrt{n}}|)_{1\leq i \leq n},
\end{align*}
we have
\begin{align*}
\MM_n\big[(Y_i')_{1\leq i \leq n}\big]&\overset{d}{=} \MM_n
\big[(\frac{1}{\sqrt{n}} |W_i^{\mu'/\sqrt{n}}|)_{1\leq i \leq n}\big],\\
\MM_n\big[(Y_i'')_{1\leq i \leq n}\big]&\overset{d}{=} \MM_n
\big[(\frac{1}{\sqrt{n}} |W_i^{\mu''/\sqrt{n}}|)_{1\leq i \leq n}\big].
\end{align*}
 By (\ref{eq09091}), it follows that
\begin{align}\label{eq09098}
\MM_n
\big[(\frac{1}{\sqrt{n}} |W_i^{\mu'/\sqrt{n}}|)_{1\leq i \leq n}\big] \preceq_{st} 
\MM_n\big[(\frac{1}{\sqrt{n}} |W_i^{\mu''/\sqrt{n}}|)_{1\leq i \leq n}\big].
\end{align}
 By (\ref{eq09097}), as $n \to \infty$,
the left and right sides of (\ref{eq09098}) converge weakly 
to $(|B_t^{\mu'}|)_{0\leq t \leq 1}$ and  $(|B_t^{\mu''}|)_{0\leq t \leq 1}$, respectively. By \cite[Proposition 3]{KKO}, we have
$(|B_t^{\mu'}|)_{0\leq t \leq 1} \preceq_{st} (|B_t^{\mu''}|)_{0\leq t \leq 1}$
as claimed. More generally, we have
$(|B_t^{\mu'}|)_{0\leq t \leq k} \preceq_{st} (|B_t^{\mu''}|)_{0\leq t \leq k}$
for (finite) $k>0$.
We can further show that 
\begin{align} \label{eq09092}
(|B_t^{\mu'}|)_{t\geq 0} \preceq_{st} (|B_t^{\mu''}|)_{t \geq 0}\;,
\end{align}
where both sides  are random elements of the Polish space $C[0,\infty)$ (\emph{cf.} \cite{Whitt}). 
Indeed, for each $k=1,2,\dots$,
define a continuous mapping $\MG_k: C[0,\infty) \to C[0,\infty)$ by
$\MG_k[x]=y$ where $y(t)=x(\min\{t,k\})$ for  $t \in [0,\infty)$. Since 
$(|B_t^{\mu'}|)_{0\leq t \leq k} \preceq_{st} (|B_t^{\mu''}|)_{0\leq t \leq k}$,
we have
\begin{align}\label{eq09099}
\MG_k\Big[(|B_t^{\mu'}|)_{t\geq 0}\Big] \preceq_{st} \MG_k\Big[(|B_t^{\mu''}|)_{t\geq 0}\Big],
\end{align}
where both sides  are random elements of  $C[0,\infty)$. As
$k \to \infty$, the left and
right sides of (\ref{eq09099}) converge a.s. to  
$(|B_t^{\mu'}|)_{t\geq 0}$ and  $(|B_t^{\mu''}|)_{t\geq 0}$, respectively.
By \cite[Proposition 3]{KKO}, (\ref{eq09092}) follows. 
So we have shown that $(|B_t^{\mu}|)_{t \geq 0}$ is stochastically increasing in $\mu \in [0,\infty)$ (and hence decreasing in $\mu \in (-\infty,0]$).
 This result has applications in sequential analysis where the stopping
 rules of some sequential tests concerning
 the drift parameter $\mu$ of Brownian motion are of the form 
 $\tau^{\mu}=\inf\{t\geq 0: |B_t^{\mu}| \geq b(t)\}$ for some boundary $b(t)$.
 (See \emph{e.g.} \cite{Siegmund}.)  By (\ref{eq09092}),
 $\tau^{\mu}$ is stochastically decreasing in $\mu \in [0,\infty)$
 and increasing in $\mu \in (-\infty, 0]$. For the special case of flat boundary 
 (i.e. $b(t)=$ constant), this result can be strengthened with the usual stochastic
 order replaced by the likelihood ratio order  (\emph{cf.} \cite[Remark 6]{HY}).
\end{remark}

\section{Proofs of Theorems \ref{thm1}--\ref{thm4}}

We need the following lemmas. Lemma \ref{lem1} can be shown easily. Lemma \ref{lem2} follows from \cite[Theorem 2]{KKO}.
(The finite-dimensional version of Lemma \ref{lem2} can be found in \cite[Theorem 3.3.7]{MS}.)

\begin{lemma}\label{lem1}
Let  $k\geq 1$. Then

(i) $(p^{k+1}+q^{k+1})/(p^k+q^k)$ is strictly decreasing in $p \in [0,\half]$ and 
strictly increasing in $p \in [\half,1]$; and 

(ii) $(p^{k+1}+q^{k+1})/(p^k+q^k) \leq (p^{k+2}+q^{k+2})/(p^{k+1}+q^{k+1})$ for $p \in [0,1]$.
\end{lemma}

\begin{lemma}\label{lem2}
Let $(\Delta_n)_{n\geq 0}$ and $(\Gamma_n)_{n\geq 0}$ be infinite-dimensional random vectors. Suppose that
$\Delta_0 \preceq_{st} \Gamma_0$, and for $n=0,1,\dots$, 
\begin{align}\label{eq991}
\ML \big[ \Delta_{n+1}\;\big|\; \Delta_i=\delta_i, i\leq n\big] \preceq_{st} 
\ML \big[ \Gamma_{n+1}\;\big|\; \Gamma_i=\gamma_i, i\leq n\big]
\end{align}
whenever $\delta_i\leq \gamma_i$ for all $i\leq n$,
where $\ML \big[ \Delta_{n+1}\;\big|\; \Delta_i=\delta_i, i\leq n\big]$ denotes the conditional distribution
of $\Delta_{n+1}$ given the past history $(\Delta_i=\delta_i, i\leq n)$
and similarly for $\ML \big[ \Gamma_{n+1}\;\big|\; \Gamma_i=\gamma_i, i\leq n\big]$.
Then there exists a coupling $\big((\Delta_n')_{n\geq 0}, (\Gamma_n')_{n\geq 0}\big)$
of $(\Delta_n)_{n\geq 0}$ and $(\Gamma_n)_{n\geq 0}$ such that
$\Delta_n' \leq \Gamma_n'$ for all $n$ a.s. Consequently, we have
$(\Delta_n)_{n\geq 0} \preceq_{st} (\Gamma_n)_{n\geq 0}$.
\end{lemma}

\begin{proof}[Proof of Theorem \ref{thm1}]
To prove
$(|S_n^{p''}|)_{n \geq 0} \preceq_{st} (|S_n^{p'}|)_{n \geq 0}$ for $0\leq p' <p'' \leq \half$, 
by Lemma \ref{lem2},
it suffices 
to show that for $n\geq 0$,
\begin{align}
\ML \big[\;|S_{n+1}^{p''}|\;\big|\;|S_i^{p''}|=\delta_i, i\leq n\big] \preceq_{st}
\ML \big[\;|S_{n+1}^{p'}|\;\big| \;|S_i^{p'}|=\gamma_i, i\leq n\big],\label{eq998}
\end{align}
where $(\delta_0,\delta_1,\dots, \delta_{n})$
and $(\gamma_0,\gamma_1,\dots, \gamma_{n})$ are  non-negative integers satisfying 
$\delta_0=\gamma_0=0$ and  $\delta_i\leq \gamma_i$ for $1\leq i\leq n$.
(Note that it is necessary for $(\delta_0,\delta_1,\dots, \delta_{n})$
and $(\gamma_0,\gamma_1,\dots, \gamma_{n})$ to satisfy the  condition that
$|\delta_{i+1}-\delta_i|=1$ and $|\gamma_{i+1}-\gamma_i|=1$ for all $i$ in order for
$P(|S_i^{p'}|=\gamma_i, i\leq n)$ and $P(|S_i^{p''}|=\delta_i, i\leq n)$ to be positive.)

By \cite[Proposition 4.1.1, Example 4.1(D), pp.166--167]{Ross}, 
the absolute value of the Bernoulli random walk $(|S_n^p|)_{n\geq 0}$ is a Markov chain with transition probabilities
$P(|S_{n+1}^p|=1\;\big|\; |S_n^p|=0)=1$, and for $k\geq 1$,
\begin{align}
P(|S_{n+1}^p|=k+1\;\big|\; |S_n^p|=k)&=\frac{p^{k+1}+q^{k+1}}{p^k+q^k},\label{eq001}\\
P(|S_{n+1}^p|=k-1 \;\big|\; |S_n^p|=k)&=\frac{p^k q+p q^k}{p^k+q^k}=1-\frac{p^{k+1}+q^{k+1}}{p^k+q^k}.
\notag
\end{align}
Note that $|S_n^p|$ and $n$ have the same parity, and that
the conditional distribution of $|S_{n+1}^p|$ given
the past history $(|S_i^p|=k_i,\; i=0,\dots, n)$ has
a one-point support $\{1\}$ if $k_n=0$ and a two-point support $\{k_n-1, k_n+1\}$ if $k_n \geq 1$.
We are now ready to verify condition (\ref{eq998}).
Since $\gamma_{n}-\delta_{n}\geq 0$ is
necessarily an even integer, and since (\ref{eq998}) holds trivially 
for $\gamma_{n}-\delta_{n}\geq 2$, it remains to deal with the case
$\gamma_{n}=\delta_{n}=k\geq 0$. If $k=0$, then
$P(|S_{n+1}^{p'}|=1\;\big|\; |S_n^{p'}|=0)= P(|S_{n+1}^{p''}|=1 \;\big|\; |S_n^{p''}|=0)= 1$. If $k\geq 1$, 
by (\ref{eq001}), we have
\[
P(|S_n^{p'}|=k+1\;\big| \; |S_n^{p'}|=k)> P(|S_n^{p''}|=k+1\;\big|\; |S_n^{p''}|=k),
\] 
since by Lemma \ref{lem1}(i), $(p^{k+1}+q^{k+1})/(p^k+q^k)$ is strictly decreasing
in $p \in [0,\half]$. So, (\ref{eq998}) holds, and 
the proof is complete.
\end{proof}

\begin{proof}[Proof of Theorem \ref{thm2}]
Unlike $(|S_n^p|)_{n\geq 0}$, $(|U_n^p|)_{n\geq 0}$ is not a Markov chain in general, since the initial state $U_0^p \in \{+1,-1\}$ is random. However, for $p'=\half$, since $X_0,X_1^{p'},X_2^{p'},\dots$ are i.i.d.,
  $(U_n^{p'})_{n\geq 0}$ may be thought of as a Bernoulli random walk
starting from state 0 at time $-1$. It follows that  $(|U_n^{p'}|)_{n\geq 0}$ is a Markov chain
with transition probabilities given by (\ref{eq001}), i.e.,
\begin{align}\label{eq0120}
P(|U_{n+1}^{p'}|=k+1\;\big|\; |U_n^{p'}|=k)=
\begin{cases}1, &\text{for $k=0$;}\\
 \half, &\text{for $k>0$.}
\end{cases}
\end{align} 

For general $p \in [0,1]$,  while $(U_n^p)_{n\geq 0}$ has a random initial state,
once $U_m^p=0$ for some $m>0$,  $(U_{m+n}^p)_{n\geq 0}$ (conditionally) behaves like
$(S_n^p)_{n\geq 0}$.
For  given $n\geq 0$, let $(k_0,\dots, k_n)$ be a sequence of non-negative integers such that 
\begin{align}
k_0=1, \;   |k_i-k_{i-1}|=1\; \text{for  $i=1,\dots, n$}. \label{eq0121}
\end{align}
Then $0\leq k_i\leq i+1$ and $i-k_i$ is odd.
 (Note that for  $P(|U_i^p|=k_i, i=0,\dots, n)$ to be positive, it is necessary for
 $(k_0,\dots, k_n)$ to satisfy (\ref{eq0121}).)
If $k_i=0$ (for some $i<n$) and $k_n>0$, we have by (\ref{eq001}) and Lemma \ref{lem1}(i), 
\begin{align}
P(|U_{n+1}^p|=k_n+1\;\big|\; |U_i^p|=k_i, i=0,\dots, n)
&=\frac{p^{k_n+1}+ q^{k_n+1}}{p^{k_n}+ q^{k_n}} \label{eq01210}\\
&\geq \half.\notag
\end{align}
 If $k_n=0$, then
$P(|U_{n+1}^p|=k_n+1\;\big|\; |U_i^p|=k_i, i=0,\dots, n)
=1$.
If $k_i>0$ for all $i\leq n$, we have
\begin{align}
P(&|U_{n+1}^p|=k_n+1\;\big|\; |U_i^p|=k_i, i=0,\dots, n)\notag\\
&=\frac{P(|U_i^p|=k_i, i=0,\dots, n,|U_{n+1}^p|=k_n+1)}{P(|U_i^p|=k_i, i=0,\dots, n)}
\notag\\
&=\frac{P(U_i^p=k_i, i\leq n,U_{n+1}^p=k_n+1)+P(U_i^p=-k_i, i\leq n,U_{n+1}^p=-k_n-1)}
{P(U_i^p=k_i, i=0,\dots, n)+P(U_i^p=-k_i, i=0,\dots, n)}
\notag\\
&=\frac{\half (pq)^{(n-k_n+1)/2}(p^{k_n}+ q^{k_n})} {\half (pq)^{(n-k_n+1)/2}(p^{k_n-1}+ q^{k_n-1})}\notag\\
&=\frac{p^{k_n}+ q^{k_n}}{p^{k_n-1}+ q^{k_n-1}} \label{eq01211}\\
&\geq \half.\notag
\end{align}
Note that $(p^{k_n}+ q^{k_n})/(p^{k_n-1}+ q^{k_n-1})=\half$ for all $p$ if $k_n=1$
and $>\half$ for $p \neq \half$ if $k_n>1$ (by Lemma \ref{lem1}(i)). So for  all $(k_0,\dots, k_n)$ satisfying
(\ref{eq0121}), we have shown that
\begin{align}\label{eq0123}
P(|U_{n+1}^p|=k_n+1\;\big|\; |U_i^p|=k_i, i=0,\dots, n)
\begin{cases}
=1, &\text{for $k_n=0$};\\
\geq \half, &\text{for $k_n>0$}.
\end{cases}
\end{align}

We are now ready to show that the conditional distributions of $(|U_n^p|)_{n\geq 0}$
and $(|U_n^{p'}|)_{n\geq 0}$ with $p'=\half \neq p \in [0,1]$ 
satisfy condition (\ref{eq991}) of Lemma \ref{lem2}. Specifically, for $n\geq 0$, let 
$(\delta_0,\dots, \delta_n)$ and $(\gamma_0,\dots, \gamma_n)$ satisfy condition (\ref{eq0121}) and  
$0\leq \delta_i\leq \gamma_i$ for 
all $i\leq n$. We need to show that 
\begin{align}
\ML \big[\;|U_{n+1}^{p'}|\;\big|\;|U_i^{p'}|=\delta_i, i\leq n\big] \preceq_{st}
\ML \big[\; |U_{n+1}^p|\;\big|\; |U_i^p|=\gamma_i, i\leq n\big]. \label{eq997}
\end{align}
Note that $\gamma_n-\delta_n\geq 0$ is necessarily an even integer. If $\gamma_n-\delta_n\geq 2$,
(\ref{eq997}) holds trivially. If $\delta_n=\gamma_n=k\geq 0$, we have by (\ref{eq0120}) and
 (\ref{eq0123}) 
\begin{align*}
P(|U_{n+1}^{p}|=k+1\;&\big|\; |U_i^p|=\gamma_i, i=0,\dots, n) \\
&\geq 
P(|U_{n+1}^{p'}|=k+1\;\big|\; |U_i^{p'}|=\delta_i, i=0,\dots, n),
\end{align*}
establishing (\ref{eq997}). The proof is complete.
\end{proof}

\begin{remark}\label{rem7}
In general, for $0<p'<p''<\half$, it does not hold that
\begin{align}
\ML \big[\; |U_{n+1}^{p''}| \;\big|\; |U_i^{p''}|=\delta_i, i\leq n\big] \preceq_{st}
\ML \big[\; |U_{n+1}^{p'}| \;\big|\; |U_i^{p'}|=\gamma_i, i\leq n\big]\label{eq01212}
\end{align}
with $0\leq \delta_i\leq \gamma_i$ for all $i\leq n$. To see this, for $\half \neq p \in (0,1)$,
by (\ref{eq01210}) and (\ref{eq01211}),  
\begin{align*}
P(|U_3^p|=2 \;\big|\; |U_0^p|=1, |U_1^p|=0, |U_2^p|=1)&=\frac{p^2+q^2}{p+q}>\half,\\
P(|U_3^p|=2 \;\big|\; |U_0^p|=1, |U_1^p|=2, |U_2^p|=1)&=\frac{p+q}{p^0+q^0}=\half.
\end{align*}
This shows that $(|U_n^p|)_{n\geq 0}$ is not a Markov chain for $\half \neq p \in (0,1)$.
Furthermore, (\ref{eq01212}) does not hold for $(\delta_0, \delta_1, \delta_2)=(1,0,1)$
and $(\gamma_0,\gamma_1,\gamma_2)=(1,2,1)$.
Hence we cannot apply Lemma \ref{lem2} to show that 
$(|U_n^{p''}|)_{n\geq 0} \preceq_{st} (|U_n^{p'}|)_{n\geq 0}$ for $0<p'<p''<\half$.
\end{remark}

\begin{proof}[Proof of Theorem \ref{thm3}]
We first derive the conditional distributions of  $(|V_n^p|)_{n\geq 0}$ and show that it
is a (non-homogeneous) Markov chain with state space $\{1,3,\dots\}$.
Let $k_i>0, i=0,\dots, n$ be  odd integers such that $k_0=1$, and for $i=0,\dots, n-1$,
\begin{align}\label{eq996}
k_{i+1}-k_i=
\begin{cases}
0 \;\text{or $2$}, \;\text{if $k_i=1$};\\
-2 \; \text{or $2$}, \;\text{otherwise}.
\end{cases}
\end{align}
(Note that for   
$P(|V_i^p|=k_i, i=0,\dots, n)$ to be positive, it is necessary for  
 $(k_0,\dots, k_n)$ to satisfy (\ref{eq996}).)
 To derive $P(|V_i^p|=k_i, i=0,\dots, n)$, observe
 that the sample path $(|V_0^p|=k_0, |V_1^p|=k_1\dots, |V_n^p|=k_n)$ corresponds to  a pair of
sequences of increments $(X_0, X_1^p,\dots, X_n^p)
\in \{+1,-1\}^{n+1}$, each being a reflection of the other. As an example, consider $k_0=1, k_1=3, k_2=1,k_3=1, k_4=3$, for which
the  corresponding pair of sequences of increments are $(X_0,X_1^p,X_2^p,X_3^p, X_4^p)=(+1, +1 -1, -1, -1)$
and $(-1,-1,+1,+1,+1)$. Furthermore, this pair of sequences of increments correspond to 
$(V_0^p,V_1^p,V_2^p,V_3^p,V_4^p)=(1,3,1,-1,-3)$ and $(-1,-3,-1,1, 3)$.
Note that $k_2=k_3=1$ implies $V_2^p V_3^p=-1$ (i.e. 
$(V_2^p, V_3^p)=(1,-1)$ or $(-1,1)$). More generally,
$|V_i^p|=|V_{i+1}^p|=1$ implies $V_i^p V_{i+1}^p=-1$. In other words, there is a sign change
between $i$ and $i+1$ if (and only if) $|V_i^p|=|V_{i+1}^p|=1$. It is also worth noting that
the sign of $V_0^p V_n^p$ depends on the parities of $(k_n-1)/2$ and $n$. If there is no $i$ with
$k_i=k_{i+1}=1$, then there is no sign change throughout, so that $V_0^p V_n^p>0$ and
$(k_n-1)/2$ and $n$ have the same parity. If there is exactly one $i$ with
$k_i=k_{i+1}=1$, then there is only one sign change, so that $V_0^p V_n^p<0$ and
$(k_n-1)/2$ and $n$ are of opposite parity. More generally, if there is an even number of
$i \in \{0,1,\dots, n-1\}$ with $k_i=k_{i+1}=1$, then $V_0^p V_n^p>0$ and $(k_n-1)/2$ and $n$ have the same parity. If there is an odd number of  $i \in \{0,1,\dots, n-1\}$ with
$k_i=k_{i+1}=1$, then  $V_0^p V_n^p<0$ and
$(k_n-1)/2$ and $n$ are of opposite parity.

Case (i): $(k_n-1)/2$ and $n$ have the same parity.
Then 
each of the corresponding pair of sequences of increments satisfies
$V_0^p V_n^p= X_0 (X_0+2\sum_{i=1}^n X_i^p)>0$ 
 (i.e. $X_0 \;(\sum_{i=1}^n X_i^p) \geq 0$). One  
sequence of increments
has $X_0=1$ and exactly $r$ of $X_1^p,\dots, X_n^p$ being $+1$ where $r=(k_n-1)/4+n/2$,
while the other has $X_0=-1$ and exactly $r$ of  $X_1^p, \dots, X_n^p$ being $-1$. Thus 
\begin{align}
P(|V_0^p|=k_0,\dots, |V_n^p|=k_n)&=\half (pq)^{n-r}(p^{2r-n}+q^{2r-n})\notag\\
&=
\half (pq)^{n/2-(k_n-1)/4} \big[p^{(k_n-1)/2}+q^{(k_n-1)/2}\big],\notag
\end{align}
from which it follows that
\begin{align}
P(|V_{n+1}^p|=k_n+2\;&\big|\; |V_i^p|=k_i, i=0,\dots, n)\notag\\
&=\frac{P(|V_0^p|=k_0,\dots, |V_n^p|=k_n, |V_{n+1}^p|=k_n+2)}{P(|V_0^p|=k_0,\dots, |V_n^p|=k_n)}\notag\\
&=\frac{\half (pq)^{(n+1)/2-(k_n+2-1)/4} \big[p^{(k_n+2-1)/2}+q^{(k_n+2-1)/2}\big]}{\half (pq)^{n/2-(k_n-1)/4} \big[p^{(k_n-1)/2}+q^{(k_n-1)/2}\big]}\notag\\
&=\frac{p^{(k_n+1)/2}+q^{(k_n+1)/2}}{p^{(k_n-1)/2}+q^{(k_n-1)/2}},\label{eq995}
\end{align}
which equals $\half$ for all $p$ if $k_n=1$ and is (strictly) decreasing in $p \in [0,\half]$ 
if $k_n\geq 3$ (by Lemma \ref{lem1}(i)).

Case (ii): $(k_n-1)/2$ and $n$ are of opposite parity. Then 
each of the corresponding pair of sequences of increments satisfies 
 $V_0^p V_n^p= X_0 (X_0+2\sum_{i=1}^n X_i^p)<0$ (i.e. $X_0 \;(\sum_{i=1}^n X_i^p)<0$). One  
sequence of increments
has $X_0=1$ and exactly $r'$ of $X_1^p,\dots, X_n^p$ being $-1$ where $r'=(k_n+1)/4+n/2$,
while the other has $X_0=-1$ and exactly $r'$ of  $X_1^p, \dots, X_n^p$ being $+1$. Thus 
\begin{align}
P(|V_0^p|=k_0,\dots, |V_n^p|=k_n)&=\half (pq)^{n-r'}(p^{2r'-n}+q^{2r'-n})\notag\\
&=
\half (pq)^{n/2-(k_n+1)/4}(p^{(k_n+1)/2}+q^{(k_n+1)/2}),\notag
\end{align}
from which it follows that
\begin{align}
P(|V_{n+1}^p|=k_n+2\;&\big|\; |V_i^p|=k_i, i=0,\dots, n)\notag\\
&=\frac{P(|V_0^p|=k_0,\dots, |V_n^p|=k_n, |V_{n+1}^p|=k_n+2)}{P(|V_0^p|=k_0,\dots, |V_n^p|=k_n)}\notag\\
&=\frac{\half (pq)^{(n+1)/2-(k_n+2+1)/4} \big[p^{(k_n+2+1)/2}+q^{(k_n+2+1)/2}\big]}{\half (pq)^{n/2-(k_n+1)/4} \big[p^{(k_n+1)/2}+q^{(k_n+1)/2}\big]}\notag\\
&=\frac{p^{(k_n+3)/2}+q^{(k_n+3)/2}}{p^{(k_n+1)/2}+q^{(k_n+1)/2}},\label{eq994}
\end{align}
which is (strictly) decreasing in $p\in [0,\half]$. 

For  both cases (i) and (ii), 
$P(|V_{n+1}^p|=k_n+2\;\big|\; |V_i^p|=k_i, i=0,\dots, n)$ depends only on $k_n$
and is decreasing in $p \in [0,\half]$. So $(|V_n^p|)_{n\geq 0}$ is a (non-homogeneous)
Markov chain.  We are now ready to verify condition (\ref{eq991}) of Lemma \ref{lem2} so as to
establish
$(|V_n^{p''}|)_{n\geq 0} \preceq_{st} (|V_n^{p'}|)_{n\geq 0}$  for $0 \leq p'<p''\leq \half$. 
Specifically, for odd integers  $(\delta_0,\dots, \delta_n)$ and $(\gamma_0,\dots, \gamma_n)$
satisfying (\ref{eq996}) and $1\leq \delta_i\leq \gamma_i$ for all $i$, we need to show that
\begin{align}
\ML \big[\;|V_{n+1}^{p''}|\;\big|\;|V_i^{p''}|=\delta_i, i\leq n\big] \preceq_{st}
\ML \big[\;|V_{n+1}^{p'}|\;\big|\;|V_i^{p'}|=\gamma_i, i\leq n\big].\label{eq993}
\end{align}
Note that $\gamma_n-\delta_n\geq 0$ is even. If $\gamma_n-\delta_n\geq 4$, then
(\ref{eq993}) holds trivially. If $\gamma_n=\delta_n$, then (\ref{eq993})
holds since the transition probabilities in (\ref{eq995}) and (\ref{eq994}) are
decreasing in $p \in [0,\half]$ by Lemma \ref{lem1}(i). Finally to show that (\ref{eq993}) holds for
$\gamma_n=\delta_n+2$, it suffices to establish
\begin{align}\label{eq992}
P(|V_{n+1}^{p'}|=\gamma_n+2\;\big| \; |V_{n}^{p'}|=\gamma_n)\geq
P(|V_{n+1}^{p''}|=\delta_n+2\;\big| \; |V_{n}^{p''}|=\delta_n).
\end{align}
If $(\gamma_n-1)/2$ and $n$ have the same parity, then $(\delta_n-1)/2$ and $n$ are of opposite parity,
so  that by (\ref{eq995}) and (\ref{eq994}), 
\begin{align*}
P(|V_{n+1}^{p'}|=\gamma_n+2\;\big| \; |V_{n}^{p'}|=\gamma_n)
&=\frac{(p')^{(\gamma_n+1)/2}+(q')^{(\gamma_n+1)/2}}{(p')^{(\gamma_n-1)/2}+(q')^{(\gamma_n-1)/2}}\\
&=\frac{(p')^{(\delta_n+3)/2}+(q')^{(\delta_n+3)/2}}{(p')^{(\delta_n+1)/2}+(q')^{(\delta_n+1)/2}}\\
&> \frac{(p'')^{(\delta_n+3)/2}+(q'')^{(\delta_n+3)/2}}{(p'')^{(\delta_n+1)/2}+(q'')^{(\delta_n+1)/2}}
\;\;\text{(by Lemma \ref{lem1}(i))}\\
&=P(|V_{n+1}^{p''}|=\delta_n+2\;\big| \; |V_{n}^{p''}|=\delta_n).
\end{align*}
If 
$(\gamma_n-1)/2$ and $n$ are of opposite parity, then $(\delta_n-1)/2$ and $n$ have the same parity,
so  that  by (\ref{eq995}) and (\ref{eq994}), 
\begin{align*}
P(|V_{n+1}^{p'}|=\gamma_n+2\;\big| \; |V_{n}^{p'}|=\gamma_n)
&=\frac{(p')^{(\gamma_n+3)/2}+(q')^{(\gamma_n+3)/2}}{(p')^{(\gamma_n+1)/2}+(q')^{(\gamma_n+1)/2}}\\
&=\frac{(p')^{(\delta_n+5)/2}+(q')^{(\delta_n+5)/2}}{(p')^{(\delta_n+3)/2}+(q')^{(\delta_n+3)/2}}\\
&\geq \frac{(p')^{(\delta_n+1)/2}+(q')^{(\delta_n+1)/2}}{(p')^{(\delta_n-1)/2}+(q')^{(\delta_n-1)/2}}
\;\;\text{(by Lemma \ref{lem1}(ii))}\\
&\geq \frac{(p'')^{(\delta_n+1)/2}+(q'')^{(\delta_n+1)/2}}{(p'')^{(\delta_n-1)/2}+(q'')^{(\delta_n-1)/2}}
\;\;\text{(by Lemma \ref{lem1}(i))}\\
&=P(|V_{n+1}^{p''}|=\delta_n+2\;\big| \; |V_{n}^{p''}|=\delta_n).
\end{align*}
So, (\ref{eq992}) holds. The proof is complete.
\end{proof}

\begin{proof}[Proof of Theorem \ref{thm4}]
First, we show that
$(|W_n^{\mu}|)_{n\geq 0}$ is a Markov chain for $\mu \in R$. For $x_i> 0, i=1,\dots, n$, let 
$f_n(x_1,\dots,x_n)$ denote the joint density function of $(|W_1^\mu|,\dots, |W_n^\mu|)$. 
For $y_i\in R, i=1,\dots, n$, let $g_n(y_1,\dots, y_n)$ denote the joint density function
of $(W_1^\mu,\dots, W_n^\mu)$. Letting $\delta_0 x_0=0$, we have
\begin{align}
f_n(x_1,\dots, x_n)&=\sum_{\delta_i \in \{1,-1\},\; i=1,\dots, n} g_n(\delta_1 x_1,\dots, \delta_n x_n)
\notag\\
&=(2\pi)^{-n/2} \sum_{\delta_i \in \{1,-1\},\; i=1,\dots, n} \exp\big[-\half \sum_{i=1}^n (\delta_i x_i-
\delta_{i-1} x_{i-1}-\mu)^2\big]\notag\\
&=(2\pi)^{-n/2} \sum_{\delta_i \in \{1,-1\},\; i=1,\dots, n} \exp\big[-\frac{n}{2} \mu^2+ \mu \delta_n x_n -\half \sum_{i=1}^n (\delta_i x_i-\delta_{i-1} x_{i-1})^2\big]\notag\\
&=(2\pi)^{-n/2} \exp\big(-\frac{n}{2} \mu^2\big) \Big[e^{\mu x_n} h_n(x_1,\dots, x_n) 
+e^{-\mu x_n} h_n'(x_1,\dots, x_n)\Big],\label{eq400}
\end{align}
where 
\begin{align}
h_n(x_1,\dots, x_n)&=\sum_{\delta_i \in \{1,-1\}, i=1,\dots, n-1,\; \delta_n=1} \exp\big[-\half \sum_{i=1}^n (\delta_i x_i-\delta_{i-1} x_{i-1})^2\big]\label{eq401}\\
h_n'(x_1,\dots, x_n)&=\sum_{\delta_i \in \{1,-1\}, i=1,\dots, n-1, \;\delta_n=-1} \exp\big[-\half \sum_{i=1}^n (\delta_i x_i-\delta_{i-1} x_{i-1})^2\big].\notag
\end{align}
Letting $\gamma_i=-\delta_i, i=1,\dots, n$, we have
\begin{align}
h_n'(x_1,\dots, x_n)&=\sum_{\gamma_i \in \{1,-1\}, i=1,\dots, n-1, \;\gamma_n=1} \exp\big[-\half \sum_{i=1}^n (-\gamma_i x_i
+\gamma_{i-1} x_{i-1})^2\big]\notag\\
&=\sum_{\gamma_i \in \{1,-1\}, i=1,\dots, n-1, \;\gamma_n=1} \exp\big[-\half \sum_{i=1}^n (\gamma_i x_i
-\gamma_{i-1} x_{i-1})^2\big]\notag\\
&=h_n(x_1,\dots, x_n),\label{eq403}
\end{align}
so that by (\ref{eq400}),
\begin{align}\label{eq402}
f_n(x_1,\dots, x_n)=(2\pi)^{-n/2} \exp\big(-\frac{n}{2} \mu^2\big) \;h_n(x_1,\dots, x_n)\;\Big(e^{\mu x_n} +e^{-\mu x_n}\Big).
\end{align}
The conditional density of $|W_{n+1}^{\mu}|$ at 
$x_{n+1}$
given $|W_i^{\mu}|=x_i, i=1,\dots, n\;$ equals by (\ref{eq402})
\begin{align}
\frac{f_{n+1}(x_1,\dots, x_n, x_{n+1})}{f_n(x_1,\dots, x_n)}&=(2\pi)^{-1/2} e^{-\mu^2/2}
\Big(\frac{h_{n+1}(x_1,\dots, x_{n+1})}{h_n(x_1,\dots, x_n)}\Big) \Big(\frac{e^{\mu x_{n+1}}
+e^{-\mu x_{n+1}}}{e^{\mu x_n}+e^{-\mu x_n}}\Big).\label{eq404}
\end{align}
By (\ref{eq401}) and (\ref{eq403}),
\begin{align*}
h_{n+1}&(x_1,\dots, x_{n+1})\\
&=\sum_{\delta_i \in \{1,-1\}, i=1,\dots, n,\; \delta_{n+1}=1} \exp\big[-\half \sum_{i=1}^{n+1} (\delta_i x_i-\delta_{i-1} x_{i-1})^2\big]\\
&=\sum_{\delta_i \in \{1,-1\}, i=1,\dots, n-1,\; \delta_n=\delta_{n+1}=1} \exp\big[-\half \sum_{i=1}^{n+1} (\delta_i x_i-\delta_{i-1} x_{i-1})^2\big]\\
&\qquad \qquad +\sum_{\delta_i \in \{1,-1\}, i=1,\dots, n-1,\; \delta_n=-1, \;\delta_{n+1}=1} \exp\big[-\half \sum_{i=1}^{n+1} (\delta_i x_i-\delta_{i-1} x_{i-1})^2]\\
&=\exp\big[-\half (x_{n+1}-x_n)^2\big] h_n(x_1,\dots,x_n)+\exp\big[-\half (x_{n+1}+x_n)^2\big] 
h_n'(x_1,\dots,x_n)\\
&=\exp\big[-\half (x_n^2+x_{n+1}^2)\big] \big[\exp(x_n x_{n+1})+\exp(-x_n x_{n+1})\big] h_n(x_1,\dots, x_n).
\end{align*}
By (\ref{eq404}),
\begin{align}
&\frac{f_{n+1}(x_1,\dots, x_n, x_{n+1})}{f_n(x_1,\dots, x_n)} \notag\\
&=(2\pi)^{-1/2} e^{-\mu^2/2}
\exp\big[-\half (x_n^2+x_{n+1}^2)\big] \big[e^{x_n x_{n+1}}+e^{-x_n x_{n+1}}\big]
\Big(\frac{e^{\mu x_{n+1}}
+e^{-\mu x_{n+1}}}{e^{\mu x_n}+e^{-\mu x_n}}\Big).\label{eq4041}
\end{align}
Since the conditional density of $|W_{n+1}^\mu|$ depends on $x_1,\dots, x_n$ only through
$x_n$, $(|W_n^\mu|)_{n\geq 0}$
is a Markov chain. Denote by $\rho(y; x, \mu)$ the conditional density of
$|W_{n+1}^\mu|$ at $y$ given $|W_n^\mu|=x$  (the transition probability density).
By (\ref{eq4041}), we have
\begin{align}
\rho(y;x,\mu)&=(2\pi)^{-\half} e^{-\half(\mu^2+x^2+y^2)} \big[e^{xy}+e^{-xy}\big]
\Big(\frac{e^{\mu y}
+e^{-\mu y}}{e^{\mu x}+e^{-\mu x}}\Big)\notag\\
&=(2\pi)^{-\half} \Big[\frac{e^{-\half (\mu^2+x^2+y^2)}}{e^{\mu x}+e^{-\mu x}}\Big] (e^{\mu y}+e^{-\mu y})
(e^{xy}+e^{-xy}). \label{eq405}
\end{align}
It is readily shown  that for $0\leq \mu'<\mu''$ and $0<x'\leq x''$, 
\begin{align*}
\frac{e^{\mu'' y}+e^{-\mu'' y}}{e^{\mu' y}+e^{-\mu' y}}\;\;\text{and}\;\;
\frac{e^{x'' y}+e^{-x'' y}}{e^{x' y}+e^{-x' y}}\;\;\text{are both increasing in}\;\; y>0.
\end{align*}
So, by (\ref{eq405}),
\begin{align}\label{eq406}
\frac{\rho(y;x'', \mu'')}{\rho(y; x', \mu')}=K \Big[\frac{e^{\mu'' y}+e^{-\mu'' y}}{e^{\mu' y}+e^{-\mu' y}}\Big]
\Big[\frac{e^{x'' y}+e^{-x'' y}}{e^{x' y}+e^{-x' y}}\Big]\;\;\text{is increasing in}\;\; y>0,
\end{align}
 where $K=K(\mu',\mu'', x', x'')>0$ is independent of $y$.

We are now ready to prove $(|W_n^{\mu'}|)_{n\geq 0} \preceq_{st} (|W_n^{\mu''}|)_{n\geq 0}$
for $0\leq \mu'<\mu''$
by verifying  condition (\ref{eq991}) of Lemma \ref{lem2}. Specifically, for $0\leq \mu'< \mu''$ and
$0<x' \leq x''$, we need to show that
\begin{align}\label{eq407}
\ML \Big[\; |W_{n+1}^{\mu'}|\;\Big|\; |W_n^{\mu'}|=x'\;\Big] \preceq_{st} \ML \Big[\; |W_{n+1}^{\mu''}|\;
\Big|\; |W_n^{\mu''}|=x''\;\Big].
\end{align}
By (\ref{eq406}), $\ML \Big[\; |W_{n+1}^{\mu'}|\;\Big|\; |W_n^{\mu'}|=x'\;\Big]$ is smaller
than $\ML \Big[\; |W_{n+1}^{\mu''}|\;
\Big|\; |W_n^{\mu''}|=x''\;\Big]$ in the likelihood ratio order, implying (\ref{eq407}).
The proof is complete.
\end{proof}

\section{Proofs of Theorems \ref{thm5} and \ref{thm6}}

 We need the following lemma to prove Theorems  \ref{thm5} and \ref{thm6}.
 
\begin{lemma}\label{lem3}
Let  $\ell,\ell', m,m'\geq 0$ be non-negative integers such that
$\ell\leq \ell'$ and $m\leq m'$. 
Then 

(i) for $0<p<\half$, $\frac{d}{dp}\;(pq)^{\ell'} (p^m+q^m)<0$ implies 
$\frac{d}{dp} \;(pq)^{\ell} (p^{m'}+q^{m'})<0$;

(ii) for $0<p<p'\leq \half$, 
\[
\frac{(p' q')^{\ell}((p')^{m'}+(q')^{m'})}{(pq)^{\ell} (p^{m'}+q^{m'})}\leq 
\frac{(p' q')^{\ell'}((p')^{m}+(q')^{m})}{(pq)^{\ell'} (p^{m}+q^{m})},
\]
where $q=1-p$ and $q'=1-p'$.
\end{lemma}
\begin{proof}

(i)  Define, for $0 \leq p \leq \half$, 
\begin{align*}
f_{m,m'}(p)&= \frac{p^m+q^m}{p^{m'}+q^{m'}},\;\;\text{and}\;\;
 g_{k,m}(p)=(pq)^k (p^m+q^m),
 \end{align*}
where $k$ is a non-negative integer.

For $m=m'$,  $f_{m,m'}(p)=1$ for all $0\leq p \leq \half$.  For $m'>m\geq 0$,
we have
\begin{align*}
(p^{m'}+q^{m'})^2 \frac{d}{dp} f_{m,m'}(p)
&=m(p^{m-1}-q^{m-1})(p^{m'}+q^{m'}) -m' (p^m+q^m)(p^{m'-1}-q^{m'-1})\\
&=(m'-m)(q^{m+m'-1}-p^{m+m'-1})+m(p^{m-1}q^{m'}-p^{m'}q^{m-1})\\
&\qquad \qquad +m'(p^mq^{m'-1}-p^{m'-1}q^m)\\
&\geq 0 \;\;(\text{for} 
\;\; 0\leq p \leq \half),
\end{align*}
so $f_{m,m'}(p)$ is increasing in $p \in [0,\half]$.
Since $(pq)^{\ell'-\ell}$ is increasing in $p \in [0,\half]$, it follows that
\[
\frac{g_{\ell',m}(p)}{g_{\ell,m'}(p)}=(pq)^{\ell'-\ell} f_{m,m'}(p)
\]
is increasing in $p \in [0,\half]$, so that for $0<p<\half$
\begin{align*}
0&\leq \frac{d}{dp} \Big(\frac{g_{\ell',m}(p)}{g_{\ell,m'}(p)}\Big)\\
&=(g_{\ell,m'}(p))^{-2} \Big( g_{\ell,m'}(p) \frac{d}{dp} g_{\ell',m}(p)-g_{\ell',m}(p) 
\frac{d}{dp} g_{\ell, m'}(p)\Big).
\end{align*}
Thus, $\frac{d}{dp}g_{\ell',m}(p)<0$ implies 
$\frac{d}{dp} g_{\ell, m'}(p)<0$.

(ii)  Since $(pq)^{\ell'-\ell} f_{m,m'}(p)=(pq)^{\ell'-\ell} \Big(\frac{p^m+q^m}{p^{m'}+q^{m'}}\Big)$ 
is increasing in $p \in [0,\half]$, we have for $0<p<p'\leq \half$,
\[
(pq)^{\ell'-\ell} \Big(\frac{p^m+q^m}{p^{m'}+q^{m'}}\Big) \leq (p'q')^{\ell'-\ell}
\Big(\frac{(p')^m+(q')^m}{(p')^{m'}+(q')^{m'}}\Big),
\]
which is equivalent to
\[
\frac{(p' q')^{\ell}((p')^{m'}+(q')^{m'})}{(pq)^{\ell} (p^{m'}+q^{m'})}\leq 
\frac{(p' q')^{\ell'}((p')^{m}+(q')^{m})}{(pq)^{\ell'} (p^{m}+q^{m})}.
\]
The proof is complete.
\end{proof}

\begin{proof}[Proof of Theorem \ref{thm5}] 
It suffices to prove Theorem \ref{thm5} for all $b_n\geq 1$. With $c_n=b_n-1\geq 0$, it is equivalent to showing
that for each $n=1,2,\dots$, and for $c_i\geq 0, i=1,\dots, n$,
\begin{align}\label{T3}
P(|U_i^p|\leq c_i, i=1,\dots, n)\;\;\text{is increasing in}\;\; p \in [0,\half].
\end{align}
 Since $(X_0, X_1^p,\dots, X_n^p)$ specifies the initial state and increments
 of  the sample path
$$(U_0^p, U_1^p,\dots, U_n^p)=(X_0, X_0^p+X_1^p,\dots, X_0+X_1^p+\cdots+X_n^p),$$
$(X_0, X_1^p,\dots, X_n^p)$ may be identified with $(U_0^p, U_1^p\dots, U_n^p)$.  For convenience,
we  refer to $(X_0, X_1^p,\dots, X_n^p)$ as a sample path of length $n$.

Write $\underline{\omega}_n=(\omega_0,\omega_1, \dots, \omega_n)$. Let
\begin{align}
\MSO_{c_1,\dots, c_n}&=\{\underline{\omega}_n \in \{+1,-1\}^{n+1}:
|\omega_0+\omega_1+\cdots+\omega_i|\leq c_i, i=1,\dots, n\}, \label{eq3.1}\\
\MSO_{c_1,\dots, c_{n-1}}^k&=\{\underline{\omega}_n \in \{+1,-1\}^{n+1}:
|\omega_0+\omega_1+\cdots+\omega_i|\leq c_i, i=1,\dots, n-1,\notag\\
&\qquad \qquad \qquad \qquad \qquad \qquad \omega_0+\omega_1+\cdots+\omega_{n}=k\}, \label{eq3.2}\\
\MSOU_{c_1,\dots, c_{n-1}}^k&=\{\underline{\omega}_n \in 
\MSO_{c_1,\dots, c_{n-1}}^k: \omega_0=+1\}, \label{eq3.3}\\
\MSOD_{c_1,\dots, c_{n-1}}^k&=\{\underline{\omega}_n \in 
\MSO_{c_1,\dots, c_{n-1}}^k: \omega_0=-1\}. \label{eq3.4}
\end{align}
Note that $\MSO_{c_1,\dots,c_{n-1}}^k$ is a disjoint union of
$\MSOU_{c_1,\dots, c_{n-1}}^k$ and $\MSOD_{c_1,\dots, c_{n-1}}^k$.
Note also that 
\begin{align}
|\MSOU_{c_1,\dots,c_{n-1}}^k|=|\MSOD_{c_1,\dots, c_{n-1}}^{-k}|, \label{eq3.40}
\end{align}
where $|\cal{A}|$ denotes the cardinality of set $\cal{A}$.
Since $k=\omega_0+\omega_1+\cdots+\omega_n$ with all $\omega_i \in \{+1,-1\}$ implies $k-n=1 \mod 2$,
we have by (\ref{eq3.2})--(\ref{eq3.4}),
\begin{align*}
\MSO_{c_1,\dots,c_{n-1}}^k=\MSOU_{c_1,\dots,c_{n-1}}^k=
\MSOD_{c_1,\dots,c_{n-1}}^k=\emptyset, \;\;\text{if}\;\;k-n=0 \mod 2.
\end{align*}
  Moreover, for $k$ and $n$ having opposite parities with $|k|\leq n+1$,
each sample path in $\MSOU_{c_1,\dots, c_{n-1}}^k$ has probability
$\half p^{(n+k-1)/2} q^{(n-k+1)/2}$ while 
each sample path in $\MSOD_{c_1,\dots, c_{n-1}}^k$ has probability
$\half p^{(n+k+1)/2} q^{(n-k-1)/2}$. (Note that $|\MSOU_{c_1,\dots, c_{n-1}}^k|=0$ 
for $k=-n-1$ and $|\MSOD_{c_1,\dots, c_{n-1}}^k|=0$ for $k=n+1$.)
 So, for $k-n=1 \mod 2$,
\begin{align}
P&((X_0,\dots,X_n^p)\in \MSO_{c_1,\dots, c_{n-1}}^k)\notag\\
&=P((X_0,\dots,X_n^p)\in \MSOU_{c_1,\dots, c_{n-1}}^k)+P((X_0,\dots,X_n^p)\in \MSOD_{c_1,\dots, c_{n-1}}^k)\notag\\
&=\half |\MSOU_{c_1,\dots,c_{n-1}}^k|\;  p^{(n+k-1)/2} q^{(n-k+1)/2}
+\half |\MSOD_{c_1,\dots,c_{n-1}}^k|\;  p^{(n+k+1)/2} q^{(n-k-1)/2}.\label{eq3.41}
\end{align}
For odd $n$ and $c_n=0$, we have by (\ref{eq3.40})--(\ref{eq3.41}),
\begin{align*}
P(|U_i^p|&\leq c_i, i=1,\dots, n)\\
&=P(|U_i^p|\leq c_i, i=1,\dots, n-1,
U_n^p=0)\\
&=P((X_0,\dots,X_n^p) \in \MSO_{c_1,\dots,c_{n-1}}^0)\\
&=\half |\MSTU_{c_1,\dots,c_{n-1}}^0|\; p^{(n-1)/2} q^{(n+1)/2}
+ \half |\MSOD_{c_1,\dots,c_{n-1}}^0| \; p^{(n+1)/2} q^{(n-1)/2}\\
&=\half |\MSOU_{c_1,\dots, c_{n-1}}^0|\;  (pq)^{(n-1)/2},
\end{align*}
which is increasing in $p \in [0,\half]$.
 So 
 \begin{align}\label{eqextra}
 P(|U_i^p|\leq c_i, i=1,\dots, n)\;\;\text{is increasing in}\; p \in [0,\half]\;\text{for odd}\;n\;\text{and}\;\;c_n=0.
 \end{align}

For $n=1$ and $c_1=0$ or $1$, we have $P(|U_1^p|\leq c_1)=P(U_1^p=0)=\half$ for all $p$.
For $n=1$ and  $c_1\geq 2$, we have $P(|U_1^p|\leq c_1)=1$
for all $p$. So (\ref{T3}) holds for $n=1$. We proceed by induction on $n$.
Suppose (\ref{T3})  holds for $n=m$ and
all $c_i\geq 0\;(i=1,\dots, m)$ for some $m\geq 1$. We need to show
that for  $c_i\geq 0$, $i=1,\dots, m+1$,
\begin{align}\label{eq3.5}
P(|U_i^p|\leq c_i, i=1,\dots, m+1)\;\;\text{is increasing in}\;\;
p \in [0,\half].
\end{align}

Since with probability $1$, $U_i^p-U_0^p$ and $i$ have the same parity, to prove (\ref{eq3.5}), we may assume without loss of
generality that $c_i$ and $i$ have opposite parities for  $i=1,\dots, m+1$. Furthermore, for some
$1<j<m+1$, let $c_i'=c_i$ for all $i \neq j$ and $c_j'=\min\{1+j, c_j, c_{j-1}+1, c_{j+1}+1\}$.
Then it is readily seen that
\[
P(|U_i^p|\leq c_i, i=1,\dots, m+1)=P(|U_i^p|\leq c_i', i=1,\dots, m+1) \;\;\text{for all $p \in [0,1]$}.
\]
More generally, it is not difficult to see that for given
$c_i\geq 0, i=1,\dots,m+1$, 
there exist $c_i'\geq 0, i=1,\dots, m+1$ with
$c_i'-i=1 \mod 2$, $c_i'\leq 1+i$ and  $|c_{i+1}'-c_i'|=1$ for all $i$, such that
\[
P(|U_i^p|\leq c_i, i=1,\dots, m+1)=P(|U_i^p|\leq c_i', i=1,\dots, m+1) \;\;\text{for all $p \in [0,1]$}.
\]
So, to prove (\ref{eq3.5}), we assume without loss of generality that
$c_i\geq 0, i=1,\dots,m+1$ satisfy
$c_i-i=1 \mod 2$, $c_i\leq 1+i$, and $|c_{i+1}-c_i|=1$ for all $i$.

If $c_{m+1}=c_m+1$, then
\[
P(|U_i^p|\leq c_i, i=1,\dots, m+1)=P(|U_i^p|\leq c_i, i=1,\dots, m),
\]
which by the induction hypothesis, is increasing in $p \in [0,\half]$. This proves (\ref{eq3.5}) 
for $c_{m+1}=c_m+1$.

Next consider $c_{m+1}=c_m-1$. If $m$ is odd, necessarily $c_m\geq 2$. If $m$ is even and $c_m=1$,
then (\ref{eq3.5}) holds  by (\ref{eqextra}). Below we assume even $c_m \geq 2$ for odd $m$ and
odd $c_m\geq 3$ for even $m$.   Let $c_i'=c_i$ for $i\leq m-1$ and
$c_m'=c_m-2$. Note that
the event $\{|U_i^p|\leq c_i, i=1,\dots, m+1\}$ is a disjoint union
of $\{|U_i^p|\leq c_i', i=1,\dots, m \}$ and $\{|U_i^p|\leq c_i,
i=1, \dots, m-1, |U_m^p|=c_m, |U_{m+1}^p|=c_m-1\}$, implying that
\begin{align}
P(|U_i^p|\leq c_i, i=1,\dots, m+1)=P(|U_i^p|\leq c_i', i=1,\dots, m)
+\gamma,\label{eq3.50}
\end{align}
where
\begin{align}
\gamma&=P(|U_i^p|\leq c_i, i=1,\dots, m-1,|U_m^p|=c_m, |U_{m+1}^p|=c_m-1)
\notag\\
&=P(|U_i^p|\leq c_i, i=1, \dots, m-1, U_m^p=c_m, X_{m+1}^p=-1)\notag\\
&\qquad +P(|U_i^p|\leq c_i, i=1, \dots, m-1, U_m^p=-c_m, X_{m+1}^p=+1)\
\notag\\
&=q \;P(|U_i^p|\leq c_i, i=1, \dots, m-1, U_m^p=c_m)\notag\\
&\qquad + p \;P(|U_i^p|\leq c_i, i=1, \dots, m-1, U_m^p=-c_m)\notag\\
&=q\; P((X_0,\dots,X_m^p)\in \MSO_{c_1,\dots,c_{m-1}}^{c_m})
+p\; P((X_0,\dots,X_m^p)\in \MSO_{c_1,\dots,c_{m-1}}^{-c_m}).\label{eq3.6} 
\end{align}
 By (\ref{eq3.6}) together with (\ref{eq3.40})--(\ref{eq3.41}), we have
\begin{align}
\gamma
&=q\Big[\half |\MSOU_{c_1,\dots,c_{m-1}}^{c_m}|\;  p^{(m+c_m-1)/2} q^{(m-c_m+1)/2}
+\half |\MSOD_{c_1,\dots,c_{m-1}}^{c_m}|\;  p^{(m+c_m+1)/2} q^{(m-c_m-1)/2} \Big] \notag\\
&+p\Big[\half |\MSOU_{c_1,\dots,c_{m-1}}^{-c_m}|\;  p^{(m-c_m-1)/2} q^{(m+c_m+1)/2}
+\half |\MSOD_{c_1,\dots,c_{m-1}}^{-c_m}|\;  p^{(m-c_m+1)/2} q^{(m+c_m-1)/2} \Big] \notag\\
& =\half \; |\MSOU_{c_1,\dots, c_{m-1}}^{c_m}|\; 
 (pq)^{(m-c_m+3)/2} \Big[p^{c_m-2}+q^{c_m-2}\Big]\notag\\
 &\qquad \qquad+\half \; |\MSOD_{c_1,\dots, c_{m-1}}^{c_m}|\; 
 (pq)^{(m-c_m+1)/2} \Big[p^{c_m}+q^{c_m}\Big]. 
 \label{eq3.7}
\end{align}
By (\ref{eq3.50}) and (\ref{eq3.7}),
\begin{align}
P(|U_i^p|&\leq c_i, i=1,\dots, m+1)\notag\\
&=P(|U_i^p|\leq c_i', i=1,\dots, m)\notag\\
&\qquad +\half \; |\MSOU_{c_1,\dots, c_{m-1}}^{c_m}|\; 
 (pq)^{(m-c_m+3)/2} \Big[p^{c_m-2}+q^{c_m-2}\Big]\notag\\
 &\qquad +\half \; |\MSOD_{c_1,\dots, c_{m-1}}^{c_m}|\; 
 (pq)^{(m-c_m+1)/2} \Big[p^{c_m}+q^{c_m}\Big].\label{eq3.700}
 \end{align}

On the other hand, observing that
the event $\{|U_i^p|\leq c_i, i=1,\dots,m\}$ is a disjoint
union of $\{|U_i^p|\leq c_i, i=1,\dots, m+1\}$ and
$\{|U_i^p|\leq c_i, i=1,\dots, m-1, |U_m^p|=c_m, |U_{m+1}^p|=c_m+1\}$,
we have
\begin{align}
P(|U_i^p|\leq c_i, i=1,\dots, m)
=P(|U_i^p|\leq c_i, i=1\dots, m+1)+\delta, \label{eq3.8}
\end{align}
where 
\begin{align}
\delta&=P(|U_i^p|\leq c_i, i=1,\dots, m-1, |U_m^p|=c_m, |U_{m+1}^p|=c_m+1)\notag\\
&=P(|U_i^p|\leq c_i, i=1,\dots, m-1, U_m^p=c_m, X_{m+1}^p=+1)\notag\\
&\qquad + P(|U_i^p|\leq c_i, i=1,\dots, m-1, U_m^p=-c_m, X_{m+1}^p=-1)\notag\\
&=p\; P(|U_i^p|\leq c_i, i=1,\dots, m-1, U_m^p=c_m)\notag\\
&\qquad + q\; P(|U_i^p|\leq c_i, i=1,\dots, m-1, U_m^p=-c_m)\notag\\
&=p\; P((X_0,\dots,X_m^p)\in \MSO_{c_1,\dots,c_{m-1}}^{c_m})+q \; P((X_0,\dots,X_m^p)\in \MSO_{c_1,\dots,c_{m-1}}^{-c_m}).\label{eq3.9}
\end{align}
By (\ref{eq3.9}) together with (\ref{eq3.40})--(\ref{eq3.41}),
\begin{align}
\delta
&=p \Big[\half |\MSOU_{c_1,\dots,c_{m-1}}^{c_m}|\;  p^{(m+c_m-1)/2} q^{(m-c_m+1)/2}
+\half |\MSOD_{c_1,\dots,c_{m-1}}^{c_m}|\;  p^{(m+c_m+1)/2} q^{(m-c_m-1)/2} \Big] \notag\\
&+q \Big[\half |\MSOU_{c_1,\dots,c_{m-1}}^{-c_m}|\;  p^{(m-c_m-1)/2} q^{(m+c_m+1)/2}
+\half |\MSOD_{c_1,\dots,c_{m-1}}^{-c_m}|\;  p^{(m-c_m+1)/2} q^{(m+c_m-1)/2} \Big] \notag\\
&=\half \; |\MSOU_{c_1,\dots, c_{m-1}}^{c_m}|\; 
 (pq)^{(m-c_m+1)/2} \Big[p^{c_m}+q^{c_m}\Big]\notag\\
 &\qquad \qquad +\half \; |\MSOD_{c_1,\dots, c_{m-1}}^{c_m}|\; 
 (pq)^{(m-c_m-1)/2} \Big[p^{c_m+2}+q^{c_m+2}\Big]. \label{eq3.91}
 \end{align}
 (Recall that $c_m$ is assumed to be less than or equal to $1+m$. If $c_m=1+m$, then
 $|\MSOD_{c_1,\dots, c_{m-1}}^{c_m}|=0$, so that the second term on the right side of (\ref{eq3.91})
 vanishes.)
By (\ref{eq3.8}) and (\ref{eq3.91}),
\begin{align}
P(|U_i^p|&\leq c_i, i=1,\dots, m+1)\notag\\
&=P(|U_i^p|\leq c_i, i=1,\dots,m)\notag\\
&\qquad -\half \; |\MSOU_{c_1,\dots, c_{m-1}}^{c_m}|\; 
 (pq)^{(m-c_m+1)/2} \Big[p^{c_m}+q^{c_m}\Big]\notag\\
 &\qquad -\half \; |\MSOD_{c_1,\dots, c_{m-1}}^{c_m}|\; 
 (pq)^{(m-c_m-1)/2} \Big[p^{c_m+2}+q^{c_m+2}\Big].\label{eq3.70}
 \end{align}
By the induction hypothesis, for $p \in (0,\half)$,
\begin{align}
\frac{d}{dp} P(|U_i^p|\leq c_i, i=1,\dots, m)\geq 0,\;\;\text{and}\;\;
\frac{d}{dp} P(|U_i^p|\leq c_i', i=1,\dots, m)\geq 0.\label{eq3.11}
\end{align}
We claim that 
\begin{align}\label{eq3.7001}
\frac{d}{dp} P(|U_i^p|\leq c_i, i=1,\dots, m+1)\geq 0\;\;\text{for $p \in (0,\half)$}.
\end{align}
If 
\[
\frac{d}{dp} (pq)^{(m-c_m+1)/2} \Big[p^{c_m}+q^{c_m}\Big]\geq 0,
\]
then by Lemma \ref{lem3}(i), 
\[
\frac{d}{dp} (pq)^{(m-c_m+3)/2} \Big[p^{c_m-2}+q^{c_m-2}\Big]\geq 0,
\]
so that by (\ref{eq3.700}) and (\ref{eq3.11}),
we have
\[
\frac{d}{dp} P(|U_i^p|\leq c_i, i=1,\dots, m+1) \geq 0.
\]
Suppose
\[
\frac{d}{dp} (pq)^{(m-c_m+1)/2} \Big[p^{c_m}+q^{c_m}\Big]<0.
\]
For $c_m\leq m-1$, by Lemma \ref{lem3}(i), 
\[
\frac{d}{dp} (pq)^{(m-c_m-1)/2} \Big[p^{c_m+2}+q^{c_m+2}\Big]< 0,
\]
so that by (\ref{eq3.70}) and (\ref{eq3.11}), we have
\begin{align*}
\frac{d}{dp} P&(|U_i^p|\leq c_i, i=1,\dots, m+1)\geq 0.
\end{align*}
For $c_m=m+1$, since $|\MSOD_{c_1,\dots, c_{m-1}}^{c_m}|=0$, 
we have  by (\ref{eq3.70}) and (\ref{eq3.11}), 
\begin{align*}
\frac{d}{dp} P&(|U_i^p|\leq c_i, i=1,\dots, m+1)\geq 0,
\end{align*}
establishing the claim (\ref{eq3.7001}).
It follows that
$P(|U_i^p|\leq c_i, i=1,\dots, m+1)$
is increasing in $p \in [0,\half]$, proving  (\ref{eq3.5}). The induction proof of
(\ref{T3}) is complete.
\end{proof}

\begin{proof}[Proof of Theorem \ref{thm6}]

For $p=0$, we have $\MP(T_{b_1,b_2,\dots}^p=n^*)=1$ where $n^*=\min\{n\geq 1: n \geq b_n\}$.
Trivially, for $0=p<p'\leq \half$, $T_{b_1,b_2,\dots}^p$ is smaller than $T_{b_1,b_2,\dots}^{p'}$ 
in the likelihood ratio order. 
For $0<p<p'\leq \half$, the distributions of $T_{b_1,b_2,\dots}^p$ and $T_{b_1,b_2,\dots}^{p'}$
have common support.
Let $n<n'$ be two consecutive points of the common support. We need to show
\begin{align}\label{LR}
\rho(n,p,p')\leq \rho(n',p,p'),
\end{align}
where 
$\rho(n,p,p')=P(T_{b_1,b_2,\dots}^{p'}=n)/P(T_{b_1,b_2,\dots}^p=n)$ is the likelihood ratio.

By
the assumption that $b_{i+1}\leq b_i$ for all $i\geq 1$, we have  $n'=n+1$ or $n+2$.
Suppose $n'=n+1$. Let $c_i=b_i-1$ for $i \geq 1$. (Note that $c_i\geq 0$ for $i\leq n$ since $n+1$ 
belongs to the support.)
Define
$\sigma=\{k\geq b_n: |\MS_{c_1,\dots, c_{n-1}}^k|>0\}$ and $\sigma'=\{k\geq b_{n+1}: 
|\MS_{c_1,\dots, c_{n}}^k|>0\}$, where for $m=n, n+1$,
\begin{align*}
\MS_{c_1,\dots, c_{m-1}}^k=\{(\omega_1,\dots, \omega_m)\in \{+1, -1\}^m&:
|\omega_1+\dots+\omega_i|\leq c_i, i=1,\dots, m-1, \\
&\qquad \omega_1+\cdots+\omega_m=k\}.
\end{align*}
Then  $|\sigma|>0$ and $|\sigma'|>0$. Also $k>k'$ for $k \in \sigma$ and $k'\in \sigma'$.
Let $k_1$ be the smallest number in $\sigma$ and $k_2$ the largest number in $\sigma'$, so that
$k_1>k_2$. 
Note that
\begin{align}
\rho(n,p,p')&=\frac{P(T^{p'}_{b_1,b_2,\dots}=n)}{P(T^p_{b_1,b_2,\dots}=n)}\notag\\
&=\frac{\sum_{k \in \sigma} |\MS_{c_1,\dots, c_{n-1}}^k| \; (p' q')^{(n-k)/2}((p')^k+(q')^k)}
{\sum_{k \in \sigma} |\MS_{c_1,\dots, c_{n-1}}^k| \; (pq)^{(n-k)/2}(p^k+q^k)}\notag\\
&=\frac{\sum_{k \in \sigma} |\MS_{c_1,\dots, c_{n-1}}^k| \;  (pq)^{(n-k)/2}(p^k+q^k)
\Big[\frac{ (p' q')^{(n-k)/2}((p')^k+(q')^k)}{(pq)^{(n-k)/2}(p^k+q^k)}\Big]}
{\sum_{k \in \sigma} |\MS_{c_1,\dots, c_{n-1}}^k| \; (pq)^{(n-k)/2}(p^k+q^k)}\notag\\
&\leq \frac{ (p' q')^{(n-k_1)/2}((p')^{k_1}+(q')^{k_1})}{(pq)^{(n-k_1)/2}(p^{k_1}+q^{k_1})},\label{ML1}
\end{align}
where the last inequality follows from Lemma 3(ii).
Similarly,
\begin{align}
\rho(n+1,p,p')&=\frac{P(T^{p'}_{b_1,b_2,\dots}=n+1)}{P(T^p_{b_1,b_2,\dots}=n+1)}\notag\\
&=\frac{\sum_{k \in \sigma'} |\MS_{c_1,\dots, c_{n}}^k| \; (p' q')^{(n+1-k)/2}((p')^k+(q')^k)}
{\sum_{k \in \sigma'} |\MS_{c_1,\dots, c_{n}}^k| \; (pq)^{(n+1-k)/2}(p^k+q^k)}\notag\\
&=\frac{\sum_{k \in \sigma'} |\MS_{c_1,\dots, c_{n}}^k| \;  (pq)^{(n+1-k)/2}(p^k+q^k)
\Big[\frac{ (p' q')^{(n+1-k)/2}((p')^k+(q')^k)}{(pq)^{(n+1-k)/2}(p^k+q^k)}\Big]}
{\sum_{k \in \sigma'} |\MS_{c_1,\dots, c_{n}}^k| \; (pq)^{(n+1-k)/2}(p^k+q^k)}\notag\\
&\geq \frac{ (p' q')^{(n+1-k_2)/2}((p')^{k_2}+(q')^{k_2})}{(pq)^{(n+1-k_2)/2}(p^{k_2}+q^{k_2})}.\label{ML2}
\end{align}
Since $k_1>k_2$,  (\ref{LR}) follows from  Lemma 3(ii), (\ref{ML1}) and (\ref{ML2}). 
The case $n'=n+2$ can be treated similarly. The proof of Theorem \ref{thm6} is complete.

\end{proof}





Shoou-Ren Hsiau, Department of Mathematics, National Changhua University of Education, Taiwan, ROC.
Email: srhsiau@cc.ncue.edu.tw

Yi-Ching Yao, Institute of Statistical Science, Academia Sinica, Taiwan, ROC. Email: yao@stat.sinica.edu.tw

\end{document}